\documentclass[12pt]{amsart}

\usepackage[hmargin=0.8in,height=8.6in]{geometry}
\usepackage{amssymb,amsthm, times}
\usepackage{delarray,verbatim}
\usepackage{ifpdf}
\ifpdf
\usepackage[pdftex]{graphicx}
 \else
\usepackage[dvips]{graphicx}
 \fi

\usepackage{bm}

\usepackage{ifpdf}
\usepackage{color}
\definecolor{webgreen}{rgb}{0,.5,0}
\definecolor{webbrown}{rgb}{.8,0,0}
\definecolor{emphcolor}{rgb}{0.95,0.95,0.95}

\usepackage{hyperref}
\hypersetup{%
          colorlinks=true,
          linkcolor=webbrown,
          filecolor=webbrown,
          citecolor=webgreen,
          breaklinks=true}
\ifpdf \hypersetup{pdftex,
            pdfstartview=FitH, 
            bookmarksopen=true,
            bookmarksnumbered=true
} \else \hypersetup{dvips} \fi

\newcommand{\lapinv}{\Phi(q)}

\numberwithin{equation}{section}

\newtheorem{theorem}{Theorem}[section]
\newtheorem{proposition}{Proposition}[section]

\newtheorem{remark}{Remark}[section]
\newtheorem{lemma}{Lemma}[section]

\numberwithin{remark}{section} \numberwithin{proposition}{section}
\numberwithin{corollary}{section}
\newcommand {\R}{\mathbb{R}}

\newcommand {\E}{\mathbb{E}}

\newcommand{\diff}{{\rm d}}

\newcommand{\lev}{L\'{e}vy }

\title[On optimal dividends in the dual model ]{On optimal dividends in the dual model}
\thanks{This version: May 28, 2013. First version: November 30, 2012. }
\author[E. Bayraktar]{Erhan Bayraktar}
\address[E. Bayraktar]{Department of Mathematics,
University of Michigan,
530 Church Street,
Ann Arbor, MI  48109-1043, USA}
\email{erhan@umich.edu}
\author[A. E. Kyprianou]{Andreas E. Kyprianou  }
\address[A. E. Kyprianou]{Department of Mathematical Sciences, 
The University of Bath, 
Claverton Down, 
Bath BA2 7AY, 
UK}
\email{a.kyprianou@bath.ac.uk}
\author[K. Yamazaki]{Kazutoshi Yamazaki }
\address[K. Yamazaki]{Department of Mathematics,
Faculty of Engineering Science, Kansai University, Suita-shi, Osaka 564-8680, Japan}
\email{kyamazak@kansai-u.ac.jp}
\date{}
\begin{document}

\begin{abstract}
We revisit the dividend payment problem in the dual model of Avanzi et al. (\cite{MR2324568}, \cite{Avanzi_2008},  and \cite{Avanzi_2011}). Using the fluctuation theory of spectrally positive L\'{e}vy processes, we give a short exposition in which we show the optimality of barrier strategies for all such L\'{e}vy processes. Moreover, we characterize the optimal barrier using the functional inverse of a scale function. We also consider the capital injection problem of \cite{Avanzi_2011} and show that its value function has a very similar form to the one in which the horizon is the time of ruin. 
\\
\noindent \small{\textbf{Key words:} dual model; dividends; capital injections;
 spectrally positive \lev processes; scale functions.
}\\
\end{abstract}

\maketitle

\section{Introduction}

In the so-called dual model, the surplus of a company is modeled by a L\'{e}vy process with \emph{positive} jumps (\emph{spectrally positive L\'{e}vy} processes); see \cite{MR2324568}, \cite{MR2372996}, \cite{Avanzi_2008}, and \cite{Avanzi_2011}. This is an appropriate model for a company driven by inventions or discoveries. Our goal is to determine the optimal dividend strategy until the time of ruin for all spectrally positive  L\'{e}vy processes. 

In \cite{Avanzi_2008}, Avanzi and Gerber   consider the dividend payment problem when the L\'{e}vy process is assumed to be the sum of an independent Brownian motion and a compound Poisson process with i.i.d.\ positive hyper-exponential jumps; they determine the optimal strategy among the set of barrier strategies. (The special case in which the jumps are exponentially distributed was obtained by  \cite{MR2372996}.) 
The optimality over all admissible strategies is later shown by \cite{Avanzi_2011} using the verification approach in \cite{MR2372996}.

In this paper, using the fluctuation theory, we give a short proof of the optimality of barrier strategies for all spectrally positive L\'{e}vy processes of bounded or unbounded variation.  Moreover, the optimal barrier is characterized using a functional inverse of the \emph{scale functions}. We also consider the cash injection problem considered in \cite{Avanzi_2011}: a variant of the dividend payment problem in which the shareholders are expected to give capital injection in order to avoid ruin. We observe that the form of the value function for this problem is very similar to the first problem we consider in which the horizon is the time of ruin.

Let us describe the dividend payment problems under consideration in more specific terms. We will denote the surplus of a company by a spectrally positive \lev process $X = \left\{X_t; t \geq 0 \right\}$ whose \emph{Laplace exponent} is given by
\begin{align}
\psi(s)  := \log \E \left[ e^{-s X_1} \right] =  c s +\frac{1}{2}\sigma^2 s^2 + \int_{(0,\infty)} (e^{-s z}-1 + s z 1_{\{0 < z < 1\}}) \nu (\diff z), \quad s \in \mathbb{R} \label{laplace_spectrally_positive}
\end{align}
where $\nu$ is a \lev measure with the support $(0,\infty)$ that satisfies the integrability condition $\int_{(0,\infty)} (1 \wedge z^2) \nu(\diff z) < \infty$.  It has paths of bounded variation if and only if $\sigma = 0$ and $\int_{( 0,1)}z\, \nu(\diff z) < \infty$; in this case, we write \eqref{laplace_spectrally_positive} as
\begin{align*}
\psi(s)   =  \mathfrak{d} s +\int_{(0,\infty)} (e^{-s z}-1 ) \nu (\diff z), \quad s \in \mathbb{R} 
\end{align*}
with $\mathfrak{d} := c + \int_{( 0,1)}z\, \nu(\diff z)$.  We exclude the trivial case in which $X$ is a subordinator (i.e., $X$ has monotone paths a.s.). This assumption implies that $\mathfrak{d} > 0$ when $X$ is of bounded variation. 

Let $\mathbb{P}_x$ be the conditional probability under which $X_0 = x$ (also let $\mathbb{P} \equiv \mathbb{P}_0$), and let $\mathbb{F} := \left\{ \mathcal{F}_t: t \geq 0 \right\}$ be the filtration generated by $X$.  Using this, the drift of $X$ is given by
\begin{align}
\mu := \E [X_1]  = - \psi'(0+). \label{drift}
\end{align}
 We also assume that $\mu < \infty$ (and hence $|\psi'(0+)| < \infty$) to ensure that the problem is nontrivial.

\subsection{The dividend payment problem until the time of ruin} \label{subsection_dividend}
We first consider a control problem in which the goal is to maximize the expected net present value (NPV) of dividends until ruin.
A (dividend) \emph{strategy} $\pi := \left\{ L_t^{\pi}; t \geq 0 \right\}$ is given by a nondecreasing, right-continuous and $\mathbb{F}$-adapted process starting at zero.    Corresponding to every strategy $\pi$, we associate a controlled surplus process $U^\pi = \{U_t^\pi: t \geq 0 \}$, which is defined by
\begin{align*}
U_t^\pi := X_t - L_t^\pi, \quad t \geq 0,
\end{align*}
where $U_{0-}^{\pi}=x$ is the initial surplus and $L_{0-}^{\pi}=0$.
The time of  ruin is defined to be
\begin{align*}
\sigma^\pi := \inf \left\{ t > 0: U_t^\pi < 0 \right\}.
\end{align*}
A lump-sum payment must be smaller than the available funds and hence it is required that
\begin{align}
L_{t}^\pi - L_{t-}^\pi \leq U_{t-}^\pi,  \quad t \leq \sigma^\pi \; \; a.s.  \label{admissibility}
\end{align}
Let $\Pi$ be the set of all admissible strategies satisfying (\ref{admissibility}).
The problem is to compute, for $q > 0$,   the expected NPV of dividends until ruin
\begin{align*}
v_\pi(x) := \E_x \left[ \int_0^{\sigma^\pi} e^{-q t} \diff L_t^\pi \right], \quad \pi \in \Pi,
\end{align*}
and to obtain an admissible strategy that maximizes it, if such a strategy exists.  Hence the problem is written as
\begin{equation}\label{eq:classical-p}
  v(x):=\sup_{\pi\in \Pi}v_\pi(x), \quad x \geq 0.
\end{equation}

\subsection{Dividend payment problem with capital injections} \label{section_injection}
In this variant of the dividend payment problem, the time horizon is infinity, and the shareholders are required to inject just enough cash to keep the company alive.
  A strategy  is now a pair $\bar{\pi} := \left\{ L_t^{\bar{\pi}}, R_t^{\bar{\pi}}; t \geq 0 \right\}$ where $L^{\bar{\pi}}$ is the cumulative amount of dividends as in the classical dividend problem and $R^{\bar{\pi}}$ is again a nondecreasing, right-continuous and $\mathbb{F}$-adapted process starting at zero,  representing the cumulative amount of injected capital satisfying
\begin{align}
\int_0^\infty e^{-qt} \diff R_t^{\bar{\pi}} < \infty, \quad a.s. \label{admissibility2}
\end{align}
Assuming that $\varphi > 1$ is the cost per unit injected capital, we want to maximize
\begin{align*}
\bar{v}_{\bar{\pi}} (x) := \mathbb{E}_x \left[ \int_0^\infty e^{-q t} \diff L_t^{\bar{\pi}} - \varphi \int_0^\infty e^{-q t} \diff R_t^{\bar{\pi}}\right], \quad x \geq 0.
\end{align*}
Hence the problem is
\begin{equation*}
  \bar{v}(x):=\sup_{\bar{\pi} \in \bar{\Pi}}\bar{v}_{\bar{\pi}}(x), \quad x \geq 0,
\end{equation*}
where $\bar{\Pi}$ is the set of all admissible strategies  that satisfy (\ref{admissibility}) and (\ref{admissibility2}).

\subsection{Outline}

In this note, we give a short proof showing that for a general spectrally positive \lev process, barrier strategies are optimal for both problems, and we give a simple characterization of the optimal barriers in terms of the scale functions; see \eqref{eq:opt-threshold} and \eqref{eq:optbar-2nd}. It is interesting to note that the forms of the value functions \eqref{v_bar_a} and \eqref{v_bar_a_bailout} are the same, while the characterizations of barrier levels are in terms of different scale functions.
Also, while, in the spectrally negative model, optimal strategies may not lie in the set of barrier strategies, our results show that the dual model can be solved in general by a barrier strategy regardless of the \lev measure.  Regarding the spectrally negative \lev model, we refer the reader to \cite{Azcue_Muler_2005} for examples where barrier strategies are suboptimal and to \cite{Loeffen_2008}  for a sufficient condition for optimality.

The structure of the rest of the paper is as follows. In Section~\ref{sec:ruin}, we solve the optimal dividend problem in which the time horizon is the time of ruin. In this section, we first collect a few results about the scale functions for spectrally one-sided L\'{e}vy processes. We then construct a candidate optimal solution out of barrier strategies by $C^1$ (resp.\ $C^2$) conditions at the barrier when $X$ is of bounded (resp.\ unbounded) variation, and verify its optimality. In Section~\ref{sec:injection}, we solve the dividend payment problem with capital injections, where we follow the same plan to the one described for Section~\ref{sec:ruin}. We conclude the paper with numerical examples in Section  \ref{section_numerics}.

\section{Solution of the dividend problem until the time of ruin}\label{sec:ruin}

For the dividend problem we described in Section \ref{subsection_dividend}, a barrier strategy at level $a \geq 0$ is denoted by $\pi_a := \left\{ L_t^a; t \leq \sigma_a \right\}$ where for all $t \geq 0$
\begin{align*}
L_t^a &:= \sup_{0 \leq s \leq t} (X_s - a) \vee 0, \\
U_t^a &:= X_t - L_t^a, 
\end{align*}
and  $\sigma_a := \inf \left\{ t > 0: U_t^a < 0 \right\}$. The corresponding expected NPV of dividends becomes
\begin{align}
v_a(x) := \E_x \left[ \int_0^{\sigma_a} e^{-qt} \diff L_t^a \right], \quad 0 \leq x \leq a. \label{def_v_a}
\end{align}
Extending \eqref{def_v_a} to the whole $\R_+$,
\begin{align} \label{value_original}
v_a(x) = \left\{ \begin{array}{ll} v_a(x), & 0 \leq x \leq a, \\ x-a + v_a(a), & x > a. \end{array} \right.
\end{align}
Our objective is to show that the optimal control lies in the class of barrier strategies and to identify $a^*$ such that $v = v_{a^*}$.

\subsection{Scale functions}
Fix $q > 0$. For any spectrally positive \lev process, there exists a function called  the  \emph{q-scale function} 
\begin{align*}
W^{(q)}: \R \mapsto [0,\infty), 
\end{align*}
which is zero on $(-\infty,0)$, continuous and strictly increasing on $[0,\infty)$, and is characterized by the Laplace transform:
\begin{align}
\int_0^\infty e^{-s x} W^{(q)}(x) \diff x = \frac 1
{\psi(s)-q}, \qquad s > \lapinv,  \label{laplace_transform}
\end{align}
where
\begin{equation}
\lapinv :=\sup\{\lambda \geq 0: \psi(\lambda)=q\}. \notag
\end{equation}
Here, the Laplace exponent $\psi$ in \eqref{laplace_spectrally_positive} is known to be zero at the origin, convex on $\R_+$; therefore $\lapinv$ is well-defined and is strictly positive as $q > 0$.  We also define
\begin{align*}
Z^{(q)}(x) := 1 + q \int_0^x W^{(q)}(y) \diff y, \quad x \in \R,
\end{align*}
and its anti-derivative
\begin{align*}
\overline{Z}^{(q)}(y) := \int_0^y Z^{(q)} (z) \diff z = y + q \int_0^y \int_0^z W^{(q)} (w) \diff w \diff z, \quad y \in \R.
\end{align*}
Notice that because $W^{(q)}$ is uniformly zero on the negative half line, we have
\begin{align}
Z^{(q)}(y) = 1  \quad \textrm{and} \quad \overline{Z}^{(q)}(y) = y, \quad y \leq 0. \label{Z_negative}
\end{align}

\begin{remark} \label{remark_smoothness_zero}
\begin{enumerate}
\item If $X$ is of unbounded variation, it is known that $W^{(q)}$ is $C^1(0,\infty)$; see, e.g.,\ Chan et al.\ \cite{Chan_2009}.  Hence, $\overline{Z}^{(q)}$ is $C^2(0,\infty)$ and $C^1 (\R)$ for the bounded variation case, while it is $C^3(0,\infty)$ and $C^2 (\R)$ for the unbounded variation case.
\item Regarding the asymptotic behavior near zero, we have that
\begin{equation}\label{eq:Wq0}
W^{(q)} (0) = \left\{ \begin{array}{ll} 0, & \textrm{if $X$ is of unbounded
variation} \\ \frac 1 {\mathfrak{d}}, & \textrm{if $X$ is of bounded variation}
\end{array} \right\}, 
\end{equation}
and 
\begin{equation}\label{eq:Wqp0}
W^{(q)'} (0+) := \lim_{x \downarrow 0}W^{(q)'} (x) =
\left\{ \begin{array}{ll}  \frac 2 {\sigma^2}, & \textrm{if }\sigma > 0 \\
\infty, & \textrm{if }\sigma = 0 \; \textrm{and} \; \nu(0,\infty) = \infty \\
\frac {q + \nu(0,\infty)} {\mathfrak{d}^2}, & \textrm{if $X$ is compound Poisson}
\end{array} \right\}.
\end{equation}
\end{enumerate}
\end{remark}

\subsection{Constructing a candidate value function}
The following is a direct application of the results given in Theorem 1 of \cite{Avram_et_al_2007} (see, in particular, page 167 of this reference).
\begin{lemma}
For every $0 \leq x \leq a$,
\begin{align*}
v_a(x) = -k(a-x) + \frac {Z^{(q)}(a-x)} {Z^{(q)}(a)} k(a),
\end{align*}
where
\begin{align}
k(y) := \overline{Z}^{(q)}(y) - \frac 1 {\Phi(q)}  Z^{(q)} (y) + \frac {\psi'(0+)} q, \quad y \geq 0. \label{def_k}
\end{align}
\end{lemma}

%
%
\begin{remark} \label{remark_consistency}
Observe that $k(0)   = - \frac 1 {\Phi(q)} + \frac {\psi'(0+)} q (< 0 \textrm{ by the convexity of } \psi$ on $[0,\infty))$.
As a result,
\begin{align*}
v_a(a) =  \frac 1 {\Phi(q)} - \frac {\psi'(0+)} q  + \frac {k(a)} {Z^{(q)}(a)}, \quad a \geq 0.
\end{align*}
By \eqref{value_original},  for $x > a$, 
\begin{align*}
v_a(x) = (x-a) + \frac 1 {\Phi(q)} - \frac {\psi'(0+)} q   + \frac {k(a)} {Z^{(q)}(a)}.
\end{align*}
On the other hand, by \eqref{Z_negative}, $k(y) = y  - \frac 1 {\Phi(q)} + \frac {\psi'(0+)} q$ for any negative $y$.
Therefore, regardless of whether $a$ is larger than $x$ or not, we can write
\begin{align}
v_a(x) = -k(a-x) + \frac {Z^{(q)}(a-x)} {Z^{(q)}(a)} k(a), \quad a,x \geq 0. \label{u_a_extension}
\end{align}
\end{remark}
\begin{remark} \label{remark_asymptotics}
The function $|k(x)|$, $x \geq 0$, is uniformly bounded by $|k(0)| < \infty$, which follows from the stochastic representation of this function in \cite{Avram_et_al_2007}.
As a result, using the duality and Wiener-Hopf factorization of spectrally positive \lev processes (see, e.g.,\  pages 73-74 and 212-213 of \cite{Kyprianou_2006}),
\begin{align*}
\lim_{a \uparrow \infty}v_a(a) = \lim_{a \uparrow \infty} \left[ \frac 1 {\Phi(q)} - \frac {\psi'(0+)} q+ \frac {k(a)}{Z^{(q)}(a)} \right] = \frac 1 {\Phi(q)} - \frac {\psi'(0+)} q = \E [(S-X)_{\eta(q)}] + \E [X_{\eta(q)}] =  \E [S_{\eta(q)}],
\end{align*}
where $S_t := \sup_{0 \leq s \leq t} (X_s \vee 0)$ and $\eta(q)$ is an exponential random variable with parameter $q > 0$ that is independent of $X$. 

This asymptotic behavior is consistent with that of the expected NPV of dividends $\tilde{v}_a$,  when $X$ is a \emph{spectrally negative process}, of a given barrier strategy starting at the barrier: 
\begin{align*}
\lim_{a \uparrow \infty}\tilde{v}_a(a) = \E [S_{\eta(q)}],
\end{align*} 
which is equation (3.15) in \cite{Avram_et_al_2007}.
\end{remark}

We  note that $v_a$, for any $a \geq 0$, is clearly continuous everywhere on $[0,\infty)$ with $v_a(0) = 0$. Here, we shall examine the smoothness of $v_a$ at $x=a$ to obtain a candidate barrier level $a^*$. In particular, we will choose $a^*$ so that $v_{a^*}$ is $C^1$ for the case $X$ is of bounded variation and $C^2$ for the case $X$ is of unbounded variation.  

Fix $x \neq a$. By differentiating \eqref{u_a_extension}, we obtain that
\begin{equation}\label{u_derivatives1}
v_a'(x) = Z^{(q)}(a-x) -  q W^{(q)}(a-x) \Lambda(a), 
\end{equation}
and when $X$ is of unbounded variation (see Remark \ref{remark_smoothness_zero} (1))
\begin{equation}\label{u_derivatives2}
v_a''(x) = -q W^{(q)}(a-x) + q W^{(q)'}(a-x) \Lambda(a) ,
\end{equation}
where
\begin{align}
\Lambda(a) :=  \frac 1 {\Phi(q)}  + \frac {k(a)} {Z^{(q)}(a)}, \quad a > 0. \label{def_lambda}
\end{align}
Substituting \eqref{def_k} into \eqref{def_lambda}, we see that $\Lambda(a) = 0$ if and only if
\begin{align}
\overline{Z}^{(q)}(a) = -\frac {\psi'(0+)} q \left( = \frac \mu q \right). \label{solution_Lambda}
\end{align}
On the other hand, since $\overline{Z}^{(q)}(x)$ is strictly increasing, goes to $\infty$ as $x \uparrow \infty$ and to $-\infty$ as $x \downarrow 0$, there exits a unique solution to \eqref{solution_Lambda}.  Because $\overline{Z}^{(q)}(0)=0$, the solution is strictly positive if and only if $\mu > 0$. We will denote our candidate barrier level by
\begin{equation}\label{eq:opt-threshold}
a^* = \begin{cases} \left(\overline{Z}^{(q)}\right)^{-1}\left (\frac \mu q \right)> 0 & \text{if $\mu>0$}, \\
0   & \text{if $\mu\leq 0$}. \end{cases}
\end{equation}
The following proposition states that with this choice of barrier level, the corresponding expected NPV function \eqref{value_original} is smooth enough to apply the verification arguments addressed below.  In view of Remark \ref{remark_smoothness_zero} (1), the smoothness at barrier level $a$ is the only point of concern.

\begin{proposition} \label{lemma_fit}  Suppose $a^* > 0$.
\begin{itemize}
\item[(i)] If $X$ is of bounded variation, $v_a$ is continuously differentiable on $(0,\infty)$ if and only if $a=a^*$.
\item[(ii)] If $X$ is of unbounded variation, $v_{a}$ is continuously differentiable on $(0,\infty)$ for all $a > 0$. However, $v_a$   is twice continuously differentiable on $(0,\infty)$ if and only if $a=a^*$.
\end{itemize}
\end{proposition}

\begin{proof}

\textbf{(i)} Because $Z^{(q)}$ and $W^{(q)}$ are continuous on $\R$ and $\R \backslash \{0\}$, respectively, it is clear in view of \eqref{u_derivatives1} that the differentiability holds anywhere on $(0,\infty) \backslash \{a\}$.  In order to show for $x=a$, letting  $x \uparrow a$ in \eqref{u_derivatives1}, 
\[
v_a'(a-) = 1 - q W^{(q)}(0) \Lambda(a).
\]
Since, when $X$ is of bounded variation $W^{(q)}(0)  \neq 0$ (see \eqref{eq:Wq0}), $v_a'(a-)=1$ only when $\Lambda(a)=0$, which happens only when $a=a^*$.

\textbf{(ii)} When $X$ is of unbounded variation $W^{(q)}(0)=0$, therefore $v_a'(a-)=1$ for all $a>0$.  The differentiability on $(0,\infty) \backslash\{a\}$ is clear similarly to \textbf{(i)}.

Regarding the twice differentiability, because $W^{(q)}$ and $W^{(q)'}$ are continuous on $\R$ and $\R \backslash \{0\}$, respectively, it is clear in view of \eqref{u_derivatives2} that the twice differentiability holds anywhere on $(0,\infty) \backslash \{a\}$.    On the other hand,
\begin{equation}\label{eq:secder}
v_{a^*}''(x)=-q W^{(q)}(a^*-x),
\end{equation}
from which it follows that $v_{a^*}''(a^*-)=0$ since $W^{(q)}(0)=0$. For any other choice of $a$, $v''_a(a-) \neq 0$, which follows from \eqref{eq:Wqp0} and \eqref{u_derivatives2}.
\end{proof}

We shall show below that $a^*$, as determined in \eqref{eq:opt-threshold}, is indeed the optimal barrier level and  \eqref{value_original} with $a=a^*$, which can be written as
\begin{equation}\label{eq:va-star}
v_{a^*}(x) =
\begin{cases}
 - \overline{Z}^{(q)} (a^*-x) - \frac {\psi'(0+)} q = - \overline{Z}^{(q)} (a^*-x) + \frac \mu  q, & \text{if $\mu>0$}, 
 \\ x, & \text{if $\mu \leq 0$},
 \end{cases}  
 \end{equation}
for any $x \geq 0$,  is the value function of the dividend payment problem.

\subsection{Verification}
By Remark \ref{remark_smoothness_zero} (1) and Lemma \ref{lemma_fit}, $v_{a^*}$ defined in \eqref{eq:va-star} is $C^2(0,\infty)$ (resp.\ $C^1(0,\infty)$) when $X$ is of unbounded (resp.\ bounded) variation.  Moreover, it is clear that $v_{a^*}(0) = 0$ in both cases. Therefore, we can use Proposition 4 of  \cite{Avram_et_al_2007}, which is a generic verification theorem for the dividend payment problems of any \lev process. (Also see Lemma 3.1 of \cite{MR2372996}.) From this theorem it follows that to prove the optimality of $v_{a^*}$ it is sufficient to demonstrate the following variational inequality:
\begin{align}
\max \left\{ (\mathcal{L}-q) v_{a^*}(x), 1 - v_{a^*}'(x) \right\} = 0, \quad x > 0. \label{veriational}
\end{align}
Here $\mathcal{L}$ is the infinitesimal generator associated with
the process $X$ applied to a sufficiently smooth function $f$
\begin{align*}
\mathcal{L} f(x) &:= -c f'(x) + \frac 1 2 \sigma^2 f''(x) + \int_0^\infty \left[ f(x+z) - f(x) -  f'(x) z 1_{\{0 < z < 1\}} \right] \nu(\diff z).
\end{align*}
We show that $v_{a^*}$ indeed satisfies \eqref{veriational} and its optimality over all admissible strategies $\Pi$.

\begin{theorem} \label{theorem_dividends}We have $v = v_{a^*}$ as defined in \eqref{eq:va-star} and $\pi_{a^*}$ is the optimal strategy over $\Pi$.
\end{theorem}
\begin{proof}
We will verify that $v_{a^*}$ satisfies \eqref{veriational} in four steps: \hfill \\
\textbf{Step 1.} Suppose $a^* > 0$.   By  Lemma \ref{lemma_fit}, it is clear that $v_{a^*}'(a^*)=1$.   Moreover, by \eqref{eq:secder}, we have that
$v_{a^*}''(x)< 0$, for $x \in (0,a^*)$. Hence $v_{a^*}'(x)$ is decreasing on $(0,a^*)$.  This shows, for $0 < x < a^*$, we have $1 - v'_{a^*}(x) \leq 0$.

\noindent \textbf{Step 2.} Again suppose $a^* > 0$.  Because of our assumption that $\psi'(0+) > -\infty$, Proposition 2 of \cite{Avram_et_al_2007} implies that, with $\widetilde{X} := -X$ and $g(x) := \overline{Z}^{(q)} (x) + \psi'(0+)/q$,  the process $\{ e^{-q (t \wedge T_{(0,a^*)})} g(\widetilde{X}_{t \wedge T_{(0,a^*)}}); t \geq 0 \}$ for $T_{(0,a^*)} := \inf \{ t > 0: \widetilde{X}_t \notin (0,a^*) \}$ is a martingale.  Now, thanks to the smoothness of $\overline{Z}^{(q)}$ as in Remark  \ref{remark_smoothness_zero} (1), It\^{o}'s lemma applies.  In particular, following the same line of arguments presented in Section 4 of \cite{Biffis_Kyprianou_2010}, this implies that $\int_0^{t \wedge T_{(0,a^*)}} e^{-qs} (\mathcal{L}-q) g(\widetilde{X}_s) \diff s = 0$, $t \geq 0$ a.s.  Hence we must have $(\mathcal{L}-q)  g(x)  = 0$ for $0 < x < a^*$. In view of \eqref{eq:va-star}, we have $(\mathcal{L}-q) v_{a^*}(x) = 0$ for all $0 < x < a^*$.  

\noindent \textbf{Step 3.} For $x \geq a^*$, by \eqref{value_original}, we have $1 - v'_{a^*}(x) = 0$.

\noindent \textbf{Step 4.} Suppose $a^* > 0$. Thanks to the smoothness of $v_{a^*}$ at $x=a^*$, which we proved in Proposition \ref{lemma_fit}, Step 2 implies that $(\mathcal{L}-q) v_{a^*}(a^*)=0$. Due to the form of $v_{a^*}$ on $x  \geq a^*$ as in \eqref{value_original}, $\mathcal{L} v_{a^*}(x) $ is a constant.
On the other hand, $q v_{a^*}(x)$ is increasing in $x$.  Hence $(\mathcal{L}-q) v_{a^*}(x) $ is decreasing on $[a^*,\infty)$ and it follows that $(\mathcal{L}-q) v_{a^*}(x) \leq 0$ for $x \geq a^*$.

Now suppose $a^* = 0$ (thus $\mu \leq 0$).  Then $f(x) := v_{a^*}(x) = x$ and  $(\mathcal{L}-q) v_{a^*}(x) = (\mathcal{L}-q) f(x) = \mathcal{L} f(x) - q x$, which is bounded from above by $0$ because $\mathcal{L} f(x) = \mu \leq 0$ for any $x  \geq  0$.
\end{proof}

\section{Solution of the dividend problem with capital injection}\label{sec:injection}

For the capital injection problem as defined in Section \ref{section_injection}, we consider the doubly reflected \lev process with upper barrier $b > 0$ and lower barrier $0$ of the form
\begin{align*}
V_t^b := X_t - L_t^b + R_t^0, \quad t \geq 0.
\end{align*}
 As shown by \cite{Pistorius_2003},  this is a Markov process taking values only on $[0,b]$. By modifying Theorem 1 of \cite{Avram_et_al_2007}, for any $b > 0$ and $0 \leq x \leq b$, we obtain that
\begin{align*}
\mathbb{E}_x \left[ \int_0^\infty e^{-q t} \diff L_t^b \right] &= - \overline{Z}^{(q)}(b-x) - \frac {\psi'(0+)} q + \frac {Z^{(q)}(b)} {q W^{(q)}b)} Z^{(q)} (b-x), \\
\mathbb{E}_x \left[ \int_0^\infty e^{-q t} \diff R_t^0 \right] &= \frac {Z^{(q)}(b-x)} {q W^{(q)}(b)}.
\end{align*}
Hence the expected payoff corresponding to the strategy $\bar{\pi}_b := \left\{ L^b_t, R^0_t; t \geq 0\right\} \in \bar{\Pi}$ is
\begin{align}
\bar{v}_b(x) :=  - \overline{Z}^{(q)}(b-x) - \frac {\psi'(0+)} q + \frac {Z^{(q)}(b) - \varphi} {q W^{(q)}(b)} Z^{(q)} (b-x), \quad 0 \leq x \leq b. \label{v_bar_a}
\end{align}
Similarly to our observations in Remark  \ref{remark_consistency}, using \eqref{Z_negative}, \eqref{v_bar_a} holds even when $x>b$.  Finally, we extend it to the negative line so that
\begin{align}
\bar{v}_b(x)  = \varphi x + \bar{v}_b(0), \quad x < 0. \label{bar_v_negative}
\end{align}

\begin{remark}
Since $Z^{(q)}(b) / W^{(q)}(b) \sim q/\Phi(q)$ as $b \uparrow \infty$ (see, e.g.,\ Exercise 8.5 of \cite{Kyprianou_2006}), it follows that
\begin{align*}
\bar{v}_b(b) =   - \frac {\psi'(0+)} q + \frac {Z^{(q)}(b) - \varphi} {q W^{(q)}(b)} \xrightarrow{b \uparrow \infty} - \frac {\psi'(0+)} q + \frac 1 {\Phi(q)} = \E [S_{\eta(q)}].
\end{align*}
This result complements the observation in Remark \ref{remark_asymptotics}; as $b$ increases  to $\infty$ the impact of ruin vanishes.
\end{remark}

\subsection{Ansatz and verification}  Analogously to the previous section, we choose our candidate barrier level using the $C^1$ ($C^2$) condition at the barrier.  For $x \neq b$, by taking derivatives
\begin{align} \label{v_injection_fit}
\begin{split}
\bar{v}_b'(x) &=  Z^{(q)}(b-x) - \frac {W^{(q)} (b-x)} {W^{(q)} (b)} ( Z^{(q)}(b) - \varphi ), \\
\bar{v}_b''(x) &=  -q W^{(q)}(b-x) + \frac {W^{(q)'} (b-x)} {W^{(q)} (b)} (Z^{(q)}(b) - \varphi).
\end{split}
\end{align}
Hence it is clear that the $C^1$ (resp.\ $C^2$) condition at $x=b$ for the bounded (resp.\ unbounded) variation case holds if and only if $Z^{(q)} (b) = \varphi$. Since $Z^{(q)}$ is strictly increasing on $(0,\infty)$, $Z^{(q)}(0)=1$ and $\lim_{x \to \infty}Z^{(q)}(x)=\infty$ (see e.g. Lemma 3.3 in \cite{Kuznetsov2013}), there exists a unique 
\begin{equation}\label{eq:optbar-2nd}
b^* :=  (Z^{(q)})^{-1}(\varphi)  > 0 \text{ whenever $\varphi > 1$}.    
\end{equation}
The candidate value function simplifies to 
\begin{equation} 
\bar{v}_{b^*}(x) :=  - \overline{Z}^{(q)}(b^*-x) - \frac {\psi'(0+)} q = - \overline{Z}^{(q)}(b^*-x) + \frac \mu  q.\label{v_bar_a_bailout}
\end{equation}
\begin{remark}
As $\varphi \downarrow 1$,  $b^* \downarrow 0$.  This is consistent with the observation given in page 158 of \cite{Avram_et_al_2007}.  On the other hand, $b^* \uparrow \infty$ as $\varphi \uparrow \infty$; as $\varphi$ increases, it gets more risky to pay dividends.
\end{remark}

Thanks to Remark \ref{remark_smoothness_zero} (1) and the way $b^*$ is chosen to ensure the smoothness at $b^*$, we can apply Proposition 4 (2) of  \cite{Avram_et_al_2007}, which tells us that it is sufficient to show that  $\bar{v}_{b^*}$ satisfies the following variational inequality:
\begin{align}
\max \left\{ (\mathcal{L}-q) \bar{v}_{b^*}(x), 1 - \bar{v}_{b^*}'(x) \right\} = 0, \quad x > 0,  \label{veriational_injection1} \\
\bar{v}_{b^*}'(x) \leq \varphi, \quad x > 0,  \label{veriational_injection2} \\
\bar{v}_{b^*}'(x) = \varphi, \quad x < 0.  \label{veriational_injection3}
\end{align}
The steps of proving the verification are similar to the ones in Theorem \ref{theorem_dividends}.  Therefore we will only verify \eqref{veriational_injection2} and \eqref{veriational_injection3}. For $0 < x < b^*$, by \eqref{v_injection_fit}, the monotonicity of $Z^{(q)}$ and \eqref{v_bar_a_bailout} imply, $\bar{v}_{b^*}'(x) =  Z^{(q)}(b^*-x)  \in [1, \varphi]$.
For $x \geq b^*$, it is clear that $\bar{v}_{b^*}'(x) =  1 < \varphi$. Also, \eqref{veriational_injection3} is satisfied  by \eqref{bar_v_negative}.
In summary, we have the following.
\begin{theorem} \label{main_results_injection}We have $\bar{v} = \bar{v}_{b^*}$ as defined in \eqref{v_bar_a_bailout} and $\bar{\pi}_{b^*} := \{ L^{b^*}_t, R^0_t; t \geq 0 \}$ is the optimal strategy over $\bar{\Pi}$.
\end{theorem}

\section{Numerical Examples} \label{section_numerics}
We have shown that the dividend payment and cash injection problems both admit solutions written in terms of the scale function.  In order to put this in practice, the only task left to do is to compute the scale function.  There are several examples of \lev processes whose scale functions are known explicitly; see \cite{Kyprianou_2006}, \cite{Kyprianou_2008}, \cite{Hubalek_Kyprianou_2009} and \cite{Kuznetsov2013}.  In general, the scale function can be computed efficiently by inverting the Laplace transform \eqref{laplace_transform} (see \cite{Surya_2008} and \cite{Kuznetsov2013}), or alternatively it can be approximated by those of phase-type \lev processes (see \cite{Asmussen_2004} and \cite{Egami_Yamazaki_2010_2}). Here,  we shall use the latter and confirm via numerical examples the results obtained in the previous sections. 

Consider a spectrally positive \lev process of the form
\begin{equation*}
  X_t  - X_0= -\mathfrak{d} t+\sigma B_t + \sum_{n=1}^{N_t} Z_n, \quad 0\le t <\infty, 
\end{equation*}
for some $\mathfrak{d} \in \R$ and $\sigma \geq 0$.  Here $B=\{B_t; t\ge 0\}$ is a standard Brownian motion, $N=\{N_t; t\ge 0\}$ is a Poisson process with arrival rate $\lambda$, and  $Z = \left\{ Z_n; n = 1,2,\ldots \right\}$ is an i.i.d.\ sequence of phase-type-distributed random variables with representation $(m,{\bm \alpha},{\bm T})$; see \cite{Asmussen_2004}.
These processes are assumed mutually independent. Its Laplace exponent \eqref{laplace_spectrally_positive} is then
\begin{align*}
 \psi(s)   = \mathfrak{d} s + \frac 1 2 \sigma^2 s^2 + \lambda \left( {\bm \alpha} (s {\bm I} - {\bm{T}})^{-1} {\bm t} -1 \right),
 \end{align*}
which is analytic for every $s \in \mathbb{C}$ except at the eigenvalues of ${\bm T}$.  Suppose $\{ -\xi_{i,q}; i \in \mathcal{I}_q \}$ is the set of the (complex-valued) roots of the equality $\psi(s) = q$ with negative real parts, and if these are assumed distinct, then
the scale function can be written
\begin{align*} 
W^{(q)}(x) = \sum_{i \in \mathcal{I}_q} C_{i,q} \left[ e^{\Phi(q) x} - e^{-\xi_{i,q}x} \right] \quad \textrm{and} \quad W^{(q)}(x) =  \sum_{i \in \mathcal{I}_q} C_{i,q} \left[ e^{\Phi(q) x} - e^{- \xi_{i,q}x} \right]  +\frac 1 {\mathfrak{d}} e^{\Phi(q) x},
\end{align*}
for the case $\sigma > 0$ and $\sigma = 0$, respectively for some $\{ C_{i,q}; i \in \mathcal{I}_q \}$; see \cite{Egami_Yamazaki_2010_2}.  For the phase-type distribution, we use $m =6$,
\begin{align*}
&{\bm T} = \left[ \begin{array}{rrrrrr}-4.0488  &  0.0000   & 0.0000  &  0.0000 &   0.0000 &   0.0000 \\
    0.1320  & -4.0012 &  0.0000  &  0.0455 &   3.7040  & 0.0044 \\
    0.2367  &  0.8595   &-4.2831  &  0.1897   & 0.2918   & 2.3724 \\
    3.1532   & 0.0000   & 0.0000 &  -4.0229  &  0.0000  &  0.0000 \\
    0.2497  &  0.0000  &  0.0000  &  3.7024  & -4.0124  &  0.0000 \\
    0.0434   & 2.1947  &  0.0938  &  0.1704  &  0.1217 &  -4.9612 \end{array} \right] \quad  \textrm{and} \quad {\bm \alpha} =    \left[ \begin{array}{r}   0.0052 \\
    0.0659 \\
    0.7446 \\
    0.0398 \\
    0.0043 \\
    0.1403  \end{array} \right]. 
\end{align*}  
This approximates  (the absolute values of) the Gaussian distribution with mean zero and standard deviation $1$, obtained using the EM-algorithm; see \cite{Egami_Yamazaki_2010_2} for  the approximation performance of the corresponding scale function.    We also let $q=0.05$ and $\lambda = 3.5$. 

We shall first confirm the results obtained in Theorem \ref{theorem_dividends}.  We consider both the bounded and unbounded variation cases with $\sigma = 0$ and $\sigma=1$, respectively.   In Figure \ref{figure_regular}, we show the value function $v_{a^*}$ as well as the point $(a^*, v_{a^*}(a^*))$ for $\mathfrak{d} =2.0, 2.33, 2.67, 3.0$ or equivalently $\mu =  0.80,0.47, 0.13,-0.20$.  The value function as well as the value of $a^*$ decrease as $\mathfrak{d}$ increases (or $\mu$ decreases); in particular $a^*=0$ for the case $\mathfrak{d} = 3.0$ (or $\mu =  -0.20 \leq 0$).  It is also observed that the value function is smooth at $a^*$ for both bounded and unbounded variation cases.  

\begin{figure}[htbp]
\begin{center}
\begin{minipage}{1.0\textwidth}
\centering
\begin{tabular}{cc}
\includegraphics[scale=0.58]{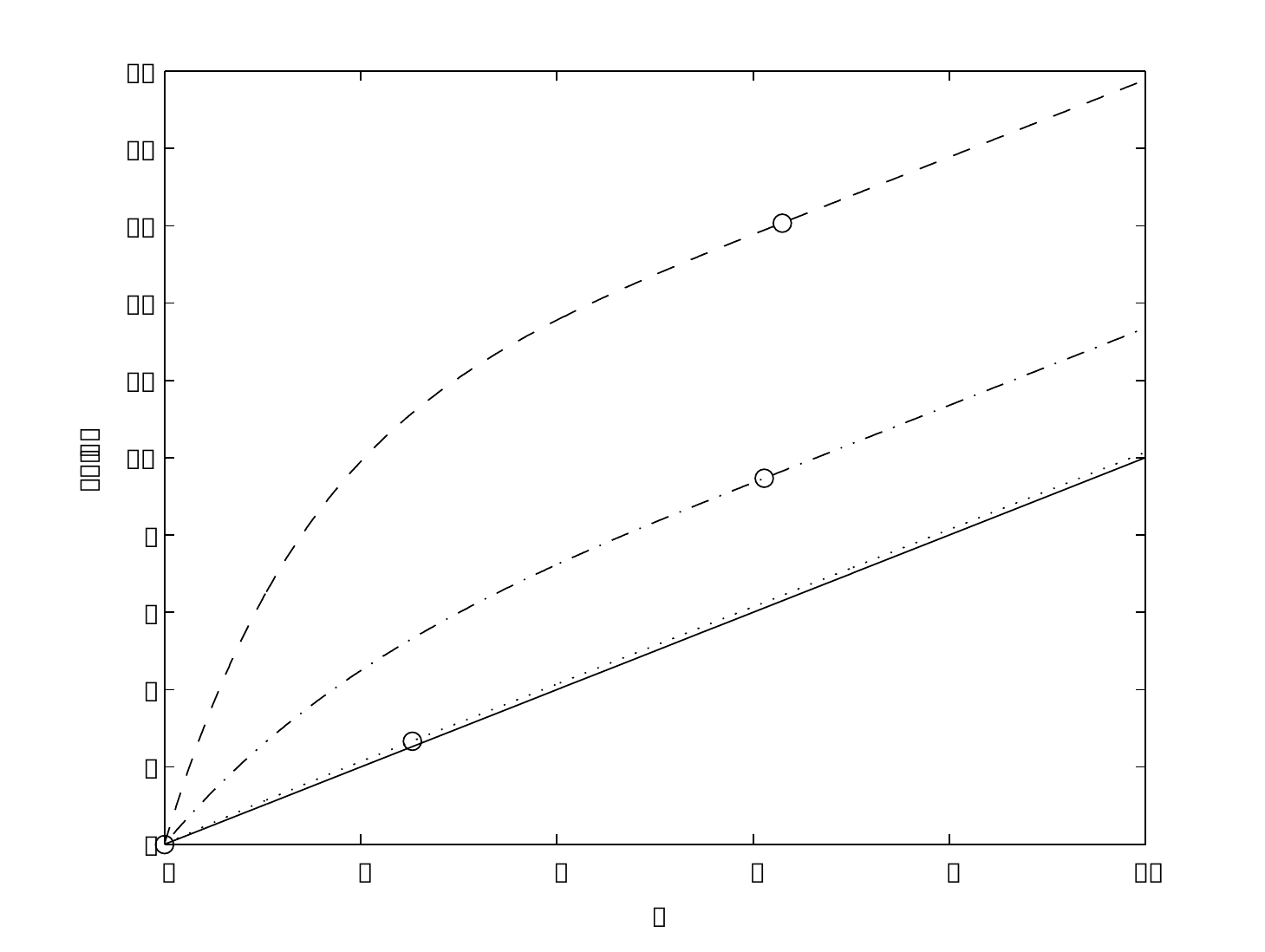}  & \includegraphics[scale=0.58]{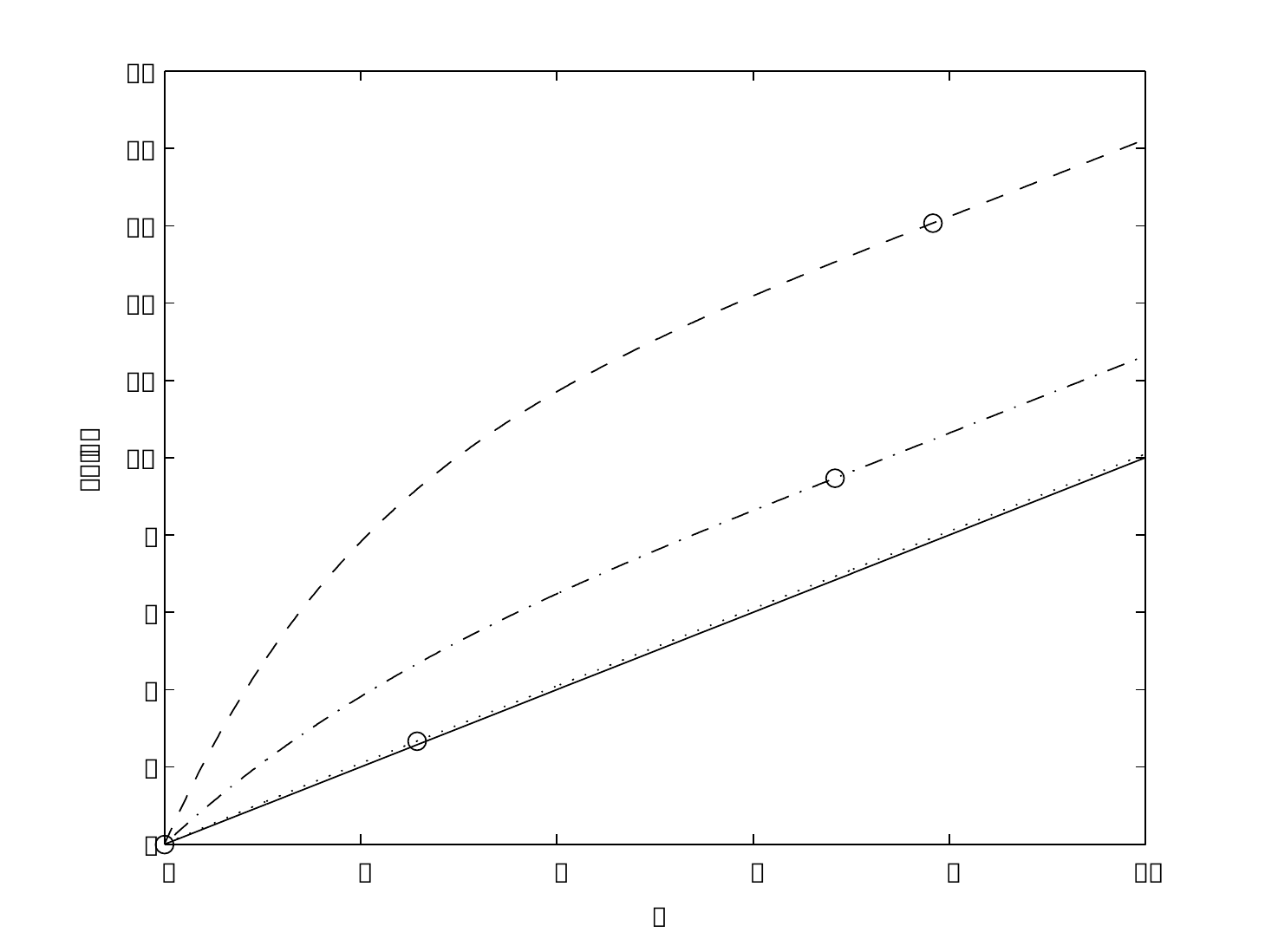} \\
$\sigma = 0$ & $\sigma = 1$ \vspace{0.3cm} \\
\end{tabular}
\end{minipage}
\caption{Results on the dividend payment problem for $\sigma = 0$ (left)  and $\sigma = 1$ (right) with $\mathfrak{d} = 2.0, 2.33, 2.67, 3.0$ (or $\mu =  0.80,0.47, 0.13,-0.20$, respectively) and a common value of $\lambda = 3.5$.} \label{figure_regular}
\end{center}
\end{figure}

Next we give results on the capital injection problem and confirm the results in Theorem  \ref{main_results_injection}.   In Figure \ref{figure_capital_injection}, we plot the value function as well as the point $(b^*, \bar{v}_{b^*}(b^*)) = (b^*, \mu/q)$ for $\sigma = 0,1$ and $\varphi= 1.001, 1.5,2,5$. Here we use the common value of  $\mathfrak{d} = 2.33$ and hence $\bar{v}_{b^*}(b^*)$ is the same for each.   The value function is indeed decreasing in the unit cost $\varphi$ and the value of $b^*$ decreases to zero as $\varphi$ decreases to $1$.  As in the case of dividend payment problem, we can again confirm the smoothness of the value function for all cases.

\begin{figure}[htbp]
\begin{center}
\begin{minipage}{1.0\textwidth}
\centering
\begin{tabular}{cc}
\includegraphics[scale=0.58]{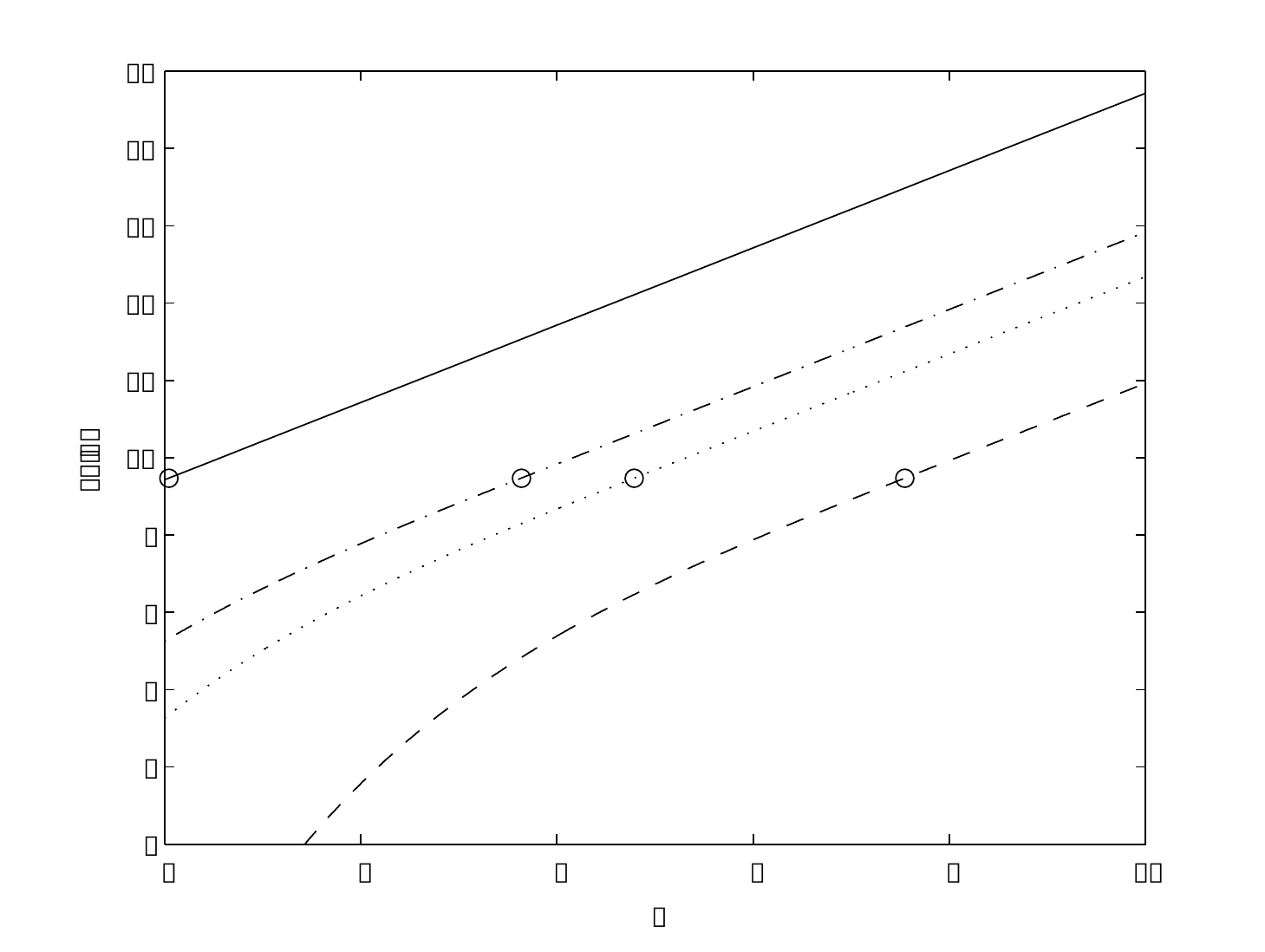}  & \includegraphics[scale=0.58]{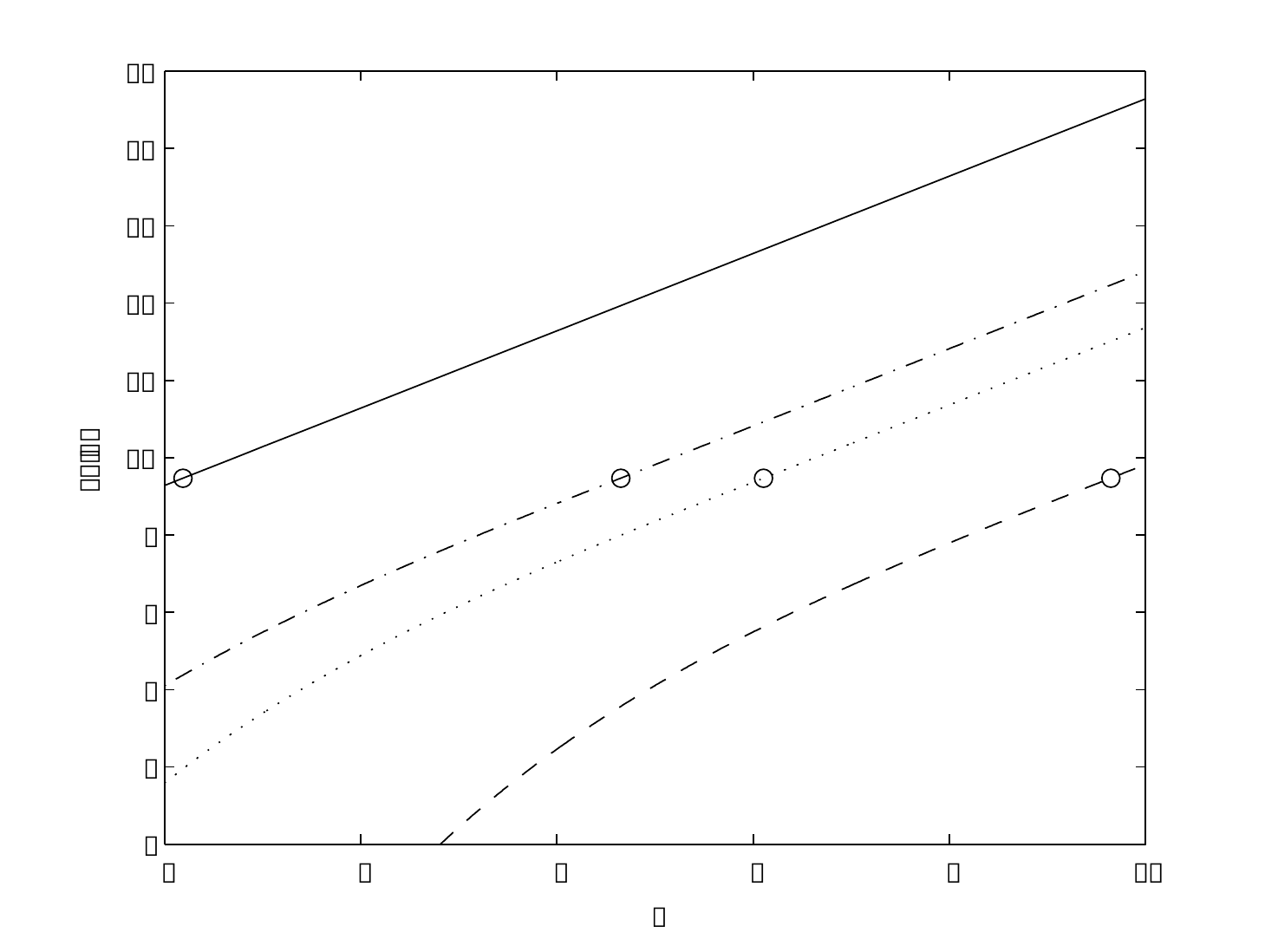} \\
$\sigma = 0$ & $\sigma = 1$ \vspace{0.3cm} \\
\end{tabular}
\end{minipage}
\caption{Results on the capital injection problem for the cases $\sigma = 0$: (left)  and $\sigma = 1$ (right) for  $\varphi= 1.001, 1.5,2,5$ with a common value $\mathfrak{d} = 2.33$ and $\lambda = 3.5$.}  \label{figure_capital_injection}
\end{center}
\end{figure}

\section*{Acknowledgements} We would like to the thank the referees and the editors for their feedback.
E. Bayraktar is supported in part by the National Science Foundation under a Career grant DMS-0955463 and in part by the Susan M. Smith Professorship.
K.\ Yamazaki is in part supported by Grant-in-Aid for Young Scientists
(B) No.\ 22710143, the Ministry of Education, Culture, Sports,
Science and Technology, and by Grant-in-Aid for Scientific Research (B) No.\  2271014, Japan Society for the Promotion of Science. A. Kyprianou would like to thank FIM (Forschungsinstitut f\"ur Mathematik) for supporting him during his sabbatical at ETH, Zurich. 
\bibliographystyle{abbrv}
\bibliographystyle{apalike}

\bibliographystyle{agsm}
\bibliography{dual_model_bib}

\end{document}